\documentclass[12pt]{amsart}
\usepackage{latexsym, amsmath, amssymb, amsthm, mathrsfs,wasysym,mathtools}
\usepackage[alwaysadjust]{paralist}
\usepackage{caption}
\usepackage{subcaption}
\usepackage[T1]{fontenc}
\usepackage{lmodern}
\usepackage[backref,colorlinks=true,linkcolor=blue,urlcolor=blue, citecolor=blue]%
  {hyperref}
\usepackage[alphabetic]{amsrefs}
\usepackage[all]{xy}
\usepackage[T1]{fontenc}
\usepackage{mathptmx}
\usepackage{extarrows}
\usepackage{microtype}
\usepackage{framed}
\usepackage{tikz}
\usepackage{tikz-cd}
\usepackage{graphicx}
\usepackage[alwaysadjust]{paralist}
\makeatletter
\providecommand\@enum@widestlabel{7}
\makeatother

\usepackage[centering, includeheadfoot, hmargin=.75in, vmargin=.75in,
  headheight=30.4pt]{geometry}
\newtheorem{lemma}{Lemma}[section]
\newtheorem{theorem}[lemma]{Theorem}

\newtheorem{proposition}[lemma]{Proposition}

\theoremstyle{definition}

\newtheorem{remark}[lemma]{Remark}

\newtheorem{problem}[lemma]{Problem}

\renewcommand{\theequation}%
{\arabic{section}.\arabic{equation}}
\newcommand{\CC}{\ensuremath{\mathbb{C}}} 
 
\newcommand{\PP}{\ensuremath{\mathbb{P}}} 
\newcommand{\QQ}{\ensuremath{\mathbb{Q}}} 
\newcommand{\RR}{\ensuremath{\mathbb{R}}} 
 
\newcommand{\sI}{\ensuremath{\kern -1pt \mathscr{I}\kern -2pt}} 
\newcommand{\sJ}{\ensuremath{\kern -2pt \mathscr{J}\kern -2pt}} 

\newcommand{\sO}{\ensuremath{\mathscr{O}}}

\newcommand{\bb}{\ensuremath{\mathfrak{b}}}

\renewcommand{\geq}{\geqslant}
\renewcommand{\leq}{\leqslant}

\DeclareMathOperator{\mult}{mult}

\DeclareMathOperator{\vol}{vol}

\newcommand{\deq}{\ensuremath{\stackrel{\textrm{def}}{=}}}

\newcommand{\Bplus}{\ensuremath{\textbf{\textup{B}}_{+} }}
\newcommand{\Bminus}{\ensuremath{\textbf{\textup{B}}_{-} }}

\newcommand{\nob}[2]{\ensuremath{\Delta_{#1}(#2)}}
\newcommand{\inob}[2]{\ensuremath{{\Delta}_{#1}(#2)}}

\input xy
\xyoption{all}

\usepackage{comment}
\usepackage{framed}
\usepackage{color}
\definecolor{shadecolor}{gray}{0.875}

\let\cal\mathcal

\let\bb\mathbb

\begin{document}

\title{Infinitesimal Newton--Okounkov bodies on products of curves}

\author{Mihai Fulger}
\address{University of Connecticut, Department of Mathematics, Storrs CT 06269, USA}
\address{Institute of Mathematics of the Romanian Academy, Bucharest 010702, Romania}
\email{mihai.fulger@uconn.edu}

\author{Victor Lozovanu}
\address{Dipartimento di Matematica, Universit\`a Degli Studi Di Genova, Genova 16146, Italy}
\email{victor.lozovanu@unige.it}

\maketitle

\begin{abstract}
	We compute the generic infinitesimal Newton--Okounkov body of a natural polarization on products of curves. This appears to be the first nontrivial full description of such a body in arbitrary dimension.
	
	%we present the first nontrivial computation of such a body in arbitrary dimension. This convex shape is related to a slice of the global Newton--Okounkov body of a toric variety with respect to a general linear flag. 
\end{abstract}

\section{Introduction}
Let $X$ be a complex projective manifold of dimension $n$ and $L$ be a big line bundle on $X$.  For any complete flag $Y_{\bullet}$ of smooth subvarieties of $X$ one associates a valuative object, an $n$-dimensional compact convex set
\[
\nob{Y_{\bullet}}{L} \ \subseteq \ \bb R^n_{+},
\]
called the \emph{Newton--Okounkov body (NObody)}. These bodies determine the numerical class of $L$ (\cite{Jow10}), and offer convex geometric interpretation to various positivity invariants such as the volume and restricted volumes of $L$ (\cite{LM09}), base loci (\cites{KL17nobodies}, \cite{KL18}) the Nakayama $\sigma$-Zariski decomposition of $L$ (\cite{Roe16}), etc. The standard references for NObodies are \cites{KK12,LM09}. 

Local positivity aspects of $L$ can be studied on the blow-up $\pi:{\rm Bl}_x\to X$, working with the line bundle $\pi^*L$, and with linear flags in the exceptional divisor $E\simeq\bb P^{n-1}$. This leads to the \emph{infinitesimal NObody (iNObody)} $\inob{x}{L}$. 
Indeed \cite{KL17} prove that $\inob{x}{L}$ determines how $x$ sits relative to base loci of $L$, as well as the Seshadri constant (also in \cite{PS21}) and Fujita--Nakayama invariant (or width) of $L$ at $x$.

When the linear flag above is sufficiently general, we obtain \emph{generic iNObodies}. These were introduced in \cite{LM09} and further studied in \cite{KL19}. In \cite{FL25gen} we show how to read even more refined invariants that may not be clearly visible for special linear flags.

NObodies are known for $T$-invariant flags on toric varieties by \cite{LM09}, and relatively well-understood on surfaces by \cite{LM09} and \cite{KLM12}. 
The computational picture is murkier for iNObodies, at least due to their connection to the Seshadri constant, a quantitative difficult invariant on its own. One expects similar difficulty when generality conditions are imposed on the flag such as when working with generic iNObodies, or on toric varieties for NObodies constructed from flags that meet the dense torus.

In this paper we construct perhaps the first examples of generic iNObodies in any dimension.

\begin{theorem}\label{thm:main}
	Let $X=C_1\times\ldots\times C_n$ be a product of $n\geq 1$ smooth complex projective curves, polarized by a box product $L=L_1\boxtimes\ldots\boxtimes L_n$ where $L_i$ have degree 1 on $C_i$. Then for any $x\in X$,
	\[
	\inob{x}{L} \ = \ \textup{simplex}\bigl\{(0,\ldots, 0), (1,1,0,\ldots,0), (2,0,2,0,\ldots,0), \ldots, (n-1,0,\ldots,0,n-1), (n,0,\ldots,0)\bigr\}\ .
\] 
\end{theorem}
The starting point of the proof goes back to an idea from \cite{LM09}. In short, it describes the vertical slice $\Delta_x(L)\cap \{t\}\times \RR^{n-1}$ for $t\in \QQ_{>0}$ as the NObody of a graded subalgebra $R(W^t_{\bullet})\subseteq \CC[y_1,\ldots ,y_n]$ with respect to generic linear flags on $\PP^{n-1}$, where $W^t_{\bullet}$ is the graded linear series of the images of the restriction maps
\[
H^0\big (\overline{X},m\cdot (\pi^*(L)-tE)\big) \ \longrightarrow \ H^0(E,\sO_{E}(mt))
\]
for $m$ sufficiently large and divisible. Asymptotically, $W_{\bullet}^t$ records the hypersurfaces that appear as tangent cones of sections of $L$ with multiplicity $t$ at $x$.
The challenge is to provide an intrinsic geometric interpretation of the family of graded linear series $W_{\bullet}^t$. When $X=C_1\times \ldots \times C_n$, intersection theory on ${\rm Bl}_xX$ implies that $W_{\bullet}^t$ is contained in a family of complete linear series associated to divisors on $X_{\Sigma}$, the blow-up of $\PP^{n-1}$ in $n$ points in general linear position. We show that this inclusion is asymptotically an equality. For this we use that $X_{\Sigma}$ is a toric variety and the asymptotic behaviour of complete linear series on them can be managed through convex geometry. 
As a consequence, the vertical slices  of $\Delta_x(L)$ are NObodies associated to a ray of divisors on $X_{\Sigma}$ with respect to a generic (strict transform) linear flag. Lastly, the global (gluing all $t\geq 0$) NObody on this special ray of divisors on $X_{\Sigma}$ is exactly the convex set described in the statement, as we are able to prove the existence of sections whose valuation are the vertices of the simplex.

For box-products of arbitrary degrees the vertical slices of $\Delta_x(L)$ are still a family of NObodies on $X_{\Sigma}$ with respect to general linear flags. But it is harder to guess the vertices. It seems an interesting problem to find an algorithm that starts with a polytope in $\RR^{n-1}$ that is a NObody on a toric variety with respect to a $T$-invariant flag, and construct another polytope, that will be a NObody on the same toric space but with respect to a generic linear flag.

In dimension $n\leq 3$, we can find non-trivial sections on $X_{\Sigma}$, whose valuations provide the vertices of a slice of the global NObody on $X_{\Sigma}$ with respect to general linear flags. More precisely:

\begin{theorem}\label{thm:smalldimensions}
Let $X=C_1\times\ldots\times C_n$ be a product of smooth complex projective curves, polarized by a box product $L=L_1\boxtimes\ldots\boxtimes L_n$ with $\deg L_i=d_i$. Assume $d_1\geq d_2\geq\ldots\geq d_n>0$. Let $x\in X$ be any point.
\begin{itemize} 
	\item If $n=2$, then $\nob{x}{L}$ is the isosceles trapezoid with vertices $(0,0)$, $ (d_2,d_2)$, $(d_1+d_2,0)$, $(d_1,d_2)$. 
	\smallskip
	
	\item If $n=3$, then $\nob{x}{L}$ is the convex hull of the 9 points 
	$(0,0,0)$, $(d_3,d_3,0)$, $(d_2,d_3,d_2-d_3)$, $(d_1,d_3,d_2-d_3)$,  $(d_1+d_2-d_3,d_3,0)$, $(d_2+d_3,0,d_2+d_3)$, $(d_1+d_3,0,d_2+d_3)$, $(d_1+d_2,0,2d_3)$,  and $(d_1+d_2+d_3,0,0)$,
\end{itemize}
In both cases,  the shapes degenerate to the simplex in Theorem \ref{thm:main} when $d_i=1$ for all $i$, while when the $d_i$ are not all equal the shapes are not simplicial.
\end{theorem}
It turns out that the verticies of our iNObodies from Theorem \ref{thm:main} and Theorem \ref{thm:smalldimensions} happen exactly when there is a change in the associated base loci of the graded linear series $W_{\bullet}^t$. 
These ideas are generalized by the authors in \cite{FL25gen} on any projetcive variety $X$ by connecting three aspects of local positivity of a line bundle $L$ on $X$ at a point $x\in X$. On one hand we have infinitesimal successive minima, then jets and finally polytopal upper and lower bounds on the generic iNObodies.
Moreover, in \cite{FL25jac} the authors study the shape of generic iNObodies on Jacobians in small dimensions. The main issue here is realizing at least asymptotically the family $W_{\bullet}^t$ as \emph{complete} graded linear series on a modification of $\PP^{n-1}$. For example, for non-hyperelliptic Jacobian threefolds, $W_{\bullet}^t$ becomes complete on a birational model of $\PP^2$ that is far from being toric, while for the hyperelliptic case we may need to go to the Zariski--Riemann space.

The paper is organized as follows. In Section 2 we review the Newton--Okounkov body construction and their basic properties. In Section 3 we use intersection theory, positivity, toric geometry and Euclidean volume computations to understand the vertical slices as Newton--Okounkov bodies on $X_{\Sigma}$. In Section 4 we treat the case of dimension $n\leq 3$. In Section 5 we use particular incomplete linear systems on $X_{\Sigma}$ to prove that the vertices of the simplex from Theorem \ref{thm:main} do belong to the iNObody, concluding its proof. In Section 6 we state and motivate some open questions.

\subsection*{Acknowledgments} We thank Sam Altschul, Aldo Conca, Rob Lazarsfeld, and Ralf Schiffler for useful conversions. The first named author was partially supported by the Simons travel grant no. 579353. The second named author was partially supported by the Research Project PRIN 2020 - CuRVI, CUP J37G21000000001, and PRIN 2022 - PEKBY, CUP J53D23003840006. Furthermore, he wishes to thank the MIUR Excellence Department of Mathematics, University of Genoa, CUP D33C23001110001. He is also a member of the INDAM-GNSAGA.

\section{Background on (i)NObodies}

\subsection{{Definition and basic properties of NObodies}}
Let $X$ be an $n$-dimensional complex projective manifold. Fix $x\in X$ and let $L$ be a big divisor/line bundle on $X$.
An \emph{admissible flag} is a complete flag 
\[Y_{\bullet}: X=Y_0\supset Y_1\supset\ldots\supset Y_n=\{x\}\] where each $Y_i$ is irreducible of codimension $i$, and smooth at $x$. 
For an effective $\bb R$-divisor $D$ on $X$, let $\nu_1(D) = {\rm ord}_{Y_1}(D)$ and inductively for each $i=2,\ldots ,n$ define 
\[
\nu_i(D)\ = \ {\rm ord}_{Y_i}((D-\nu_{i-1}(D)Y_{i-1})|_{Y_{i-1}})
\] 
Note that $(D-\nu_{i-1}(D)Y_{i-1})|_{Y_{i-1}}$ is a well-defined divisor in $Y_{i-1}$. 
Setting $\nu_{Y_{\bullet}}=(\nu_1,\ldots,\nu_n)$, we define the NObody as the topological closure of the convex body
\[\nob{Y_{\bullet}}{L}\ \deq\  \overline{\bigl\{\nu_{Y_{\bullet}}(D)\ \mid\ D\equiv L,\ D\text{ effective}\bigr\}} \ \subseteq \ \RR^n_+.\] 
The standard definition, e.g. \cite{LM09}, essentially uses $\bb Q$-linear equivalence of $\bb Q$-divisors, not numerical equivalence of $\bb R$-divisors. The definitions are equivalent for big divisors and extend naturally and continuously to big numerical $\bb R$-divisor classes.
Other important general properties:
\begin{enumerate}[(i)]
	\item ${\rm vol}_{\bb R^n}\big(\nob{Y_{\bullet}}{L}\big)=\frac 1{n!}\vol(L)$.
	\smallskip
	
\item %$\nob{Y_{\bullet}}{L}_{\nu_1\geq t} \ = \ (t,0,\ldots,0)\ + \ \nob{Y_{\bullet}}{L-tY_1}$ for all $t\geq 0$ whenever $L-tY_1$ is big. In particular 
$\nob{Y_{\bullet}}{L}$ projects on the $\nu_1$ axis as the segment $[\sigma_{Y_1}(L),\mu(L;Y_1)]$. Here $\sigma_{Y_1}(L)$ is the coefficient of $Y_1$ in the negative part $N_{\sigma}(L)$ of the Nakayama $\sigma$-Zariski decomposition (\cite{Nak04}) of $L$, while 
\[\mu(L;Y_1)\deq\sup\{t\geq 0\ \mid\  L-tY_1\text{ is big}\}.\]
	
		\item If $Y_1\not\subset\Bplus(L)$ 	and $t\geq 0$ with $L-tY_1$ big, then 
		\[\nob{Y_{\bullet}}{L}_{\nu_1=t}\subseteq\nob{Y_{\bullet}}{(L-tY_1)|_{Y_1}}.\] 
		The inclusion is often strict since the smaller subset considers only effective divisors on $Y_1$ that appear as restrictions of effective divisors on $X$. Furthermore, ${\rm Vol}_{\bb R^{n-1}}\nob{Y_{\bullet}}{L}_{\nu_1=t}=\frac 1{(n-1)!}\vol_{X|Y_1}(L-tY_1)$ with the restricted volume computed from sections of $(L-tY_1)|_{Y_1}$ that are the restriction of some section of $L-tY_1$.
		\item As the numerical class of $L$ varies inside the cone of big $\bb R$-classes in $N^1(X)_{\bb R}=N^1(X)\otimes\bb R$, the bodies $\nob{Y_{\bullet}}{L}$ glue in a convex cone in $N^1(X)_{\bb R}\times\bb R^n_{\geq 0}$ whose slice over a big class $\xi\in N^1(X)_{\bb R}$ is $\nob{Y_{\bullet}}{\xi}$. We call this cone the \emph{global NObody} on $X$ with respect to the flag $Y_{\bullet}$.
\end{enumerate}		
			
\subsection{{Definition and basic properties of iNObodies}}\label{sec:definobody}			
	Let $\pi:\overline X\to X$ denote the blow-up of $x$ with exceptional divisor $E\simeq\bb P^{n-1}$. Consider a flag $Y_{\bullet}$ in $\overline X$ such that $Y_{0}=\overline X$, and $Y_1=E\simeq\bb P^{n-1}$, and $Y_{i}$ is a linear $\bb P^{n-i}$ inside $Y_{i-1}$ for all $i>1$. We call $\nob{Y_{\bullet}}{\pi^*L}$ an \emph{infinitesimal NObody (iNObody)}.
		Assuming $x\not\in\Bminus(L)$, e.g., if $L$ is nef, then:
	
\begin{enumerate}[(i)] 
	\item $(0,\ldots,0)\in\inob{Y_{\bullet}}{\pi^*L}$. 
	\smallskip
	
	\item The projection of $\inob{Y_{\bullet}}{\pi^*L}$ on the $\nu_1$-axis is the interval $[0,\mu(L;x)]$, where 
	\[\mu(L;x)\deq\sup\{t\geq 0\ \mid\ L_t\deq\pi^*L-tE\text{ is big}\}\] 
	is the \emph{Fujita--Nakayama (or width)} of $L$ at $x$.
	\smallskip
	
	\item For all $0<t<\mu(L;x)$, we have that 
	\begin{equation}\label{eq:generalupperbound}\inob{Y_{\bullet}}{\pi^*L}_{\nu_1=t}\subseteq t\cdot\nob{Y_{\bullet}}{\cal O_{\bb P^{n-1}}(1)}.\end{equation} 
		
	\item For all \emph{very general} linear flags $Y_{\bullet}$ in $E$ (meaning outside the union of countably many proper closed subsets of the flag manifold), the NObody $\nob{Y_{\bullet}}{\pi^*L}$ is independent of $Y_{\bullet}$ and we call it the \emph{generic infinitesimal NObody (iNObody)} and denote it by $\inob{x}{L}$.
	\smallskip
	
	\item We can define a global iNObody, and a global generic iNObody.
\end{enumerate}

\section{{Vertical slices via multiplicity bounds and toric geometry}}
We prove that the vertical slices of $\inob{x}{L}$ identify with NObodies of complete linear series on the blow-up $\rho: \widetilde E\to E$ of $\bb P^{n-1}$ in the $n$ torus fixed points. The strategy is to show a double inclusion. The upper bound comes from a multiplicity estimate, via computations of positive cones. For the lower bound we draw on the toric structure of $\widetilde E$. 

Let $X=C_1\times\ldots\times C_n$ and $x=(c_1,\ldots,c_n)\in X$. Denote by $h_i$ the numerical class of the ``coordinate'' hypersurfaces $H_i={\rm pr}_i^{-1}\{c_i\}$ through $x$. By abuse of notation denote further by 
\[C_1\deq C_1\times\{(c_2,\ldots,c_n)\}\subset C_1\times\ldots\times C_n\] and similarly define $C_i\subset X$.
For $\pi:\overline X\to X$ the blow-up of $x$ with exceptional divisor $E\simeq\bb P^{n-1}$, the strict transforms $\overline C_i$ of the $C_i$ are now pairwise disjoint. 

\begin{lemma}[Positive cones on $X$ and $\overline X$]$ $
	\begin{enumerate}
		\item On $X$, the class $D=\sum_{i=1}^nd_ih_i$ is pseudoeffective if and only if $d_i\geq 0$ for all $i$.
		\smallskip
		
		\item Assume $n\geq 2$. On $\overline X$, the class of $D=\sum_{i=1}^nd_i\overline H_i+eE$ is pseudoeffective if and only if $d_i\geq 0$ for all $i$ and $e\geq 0$. In particular, for $f\geq 0$, a class $\sum_{i=1}^nc_i\pi^*h_i-fE$ is pseudoeffective if and only if $c_i\geq 0$ for all $i$ and $f\leq\sum_{i=1}^nc_i$. 
	\end{enumerate}
\end{lemma}

\noindent It follows from (2), that if $D$ is as in (1), then $\mu(D;x)=\sum_id_i$.

\begin{proof}
	For both parts one implication is trivial. It is sufficient to treat the case where $D$ is effective.
	
	(1) The curves $C_i\subseteq X$ are movable hence $D\cdot C_i\geq 0$ for all $i$, and $h_i\cdot C_j=\delta_{ij}$, hence $D\cdot C_i=d_i$ for all $i$.
	
	(2) Since $\pi_*D=\sum_id_ih_i$ is also effective, then $d_i\geq 0$ for all $i$ by (1). Without loss of generality, suppose $D$ does not have any of the $\overline H_i$ in its support. Then $D|_{\overline H_n}\equiv\sum_{i=1}^{n-1}d_i\overline H_i+eE$ is effective, in one dimension less. Conclude by induction. For the base case $n=2$, note that $\overline H_1$, $E$, $\overline H_2$ are $-1$ curves. If $C$ is an irreducible curve distinct from them, and $C\equiv d_1\overline H_1+d_2\overline H_2+eE$ with $d_1,d_2\geq 0$, then $C\cdot\overline H_1\geq 0$ forces $e\geq 0$.
\end{proof}

The following lemma provides some necessary conditions that are satisfied by the restricted linear series $W_{\bullet}^t$ from the introduction. It shows how far away they are from not being a complete series on $E=\PP^{n-1}$.

\begin{lemma}\label{lem:nonakamaye}
	Let $D\equiv\sum_{i=1}^nd_ih_i$ where $d_i\geq 0$. Denote $D_t\deq\pi^*D-tE$. Then for all $t\in[0,\ldots,d_1+\ldots+d_n]$ and for every effective divisor $D'$ numerically equivalent to $D_t$, we have
	\[\mult_{\overline C_i}D'\geq t-d_i\]
	for all $i$. 
\end{lemma}
\begin{proof}
Let $\rho:\widetilde X\to\overline X$ be the blow-up of all $\overline C_i$ with exceptional divisors $E_i$. As $\overline C_i$ is the complete intersection of the $\overline H_j$ with $j\neq i$ and $\overline H_j\equiv\pi^*H_j-E$, then $E_i\simeq C_i\times\bb P^{n-2}$ and $-E_i|_{E_i}$ represents $\cal O_{C_i}(c_i)\boxtimes\cal O_{\bb P^{n-2}}(1)$.

It is sufficient to prove that $(\rho^*D_t-mE_i)|_{E_i}$ is not pseudoeffective for $m<t-d_i$. To start note that
$D$ restricts to $C_i\subset X$ as a divisor of degree $d_i$. Then $D_t$ restricts to $\overline C_i$ as a divisor of degree $d_i-t$. Finally $\rho^*D_t-mE_i$ restricts to $E_i\simeq C_i\times\bb P^{n-2}$ as a bidegree $(d_i-t+m,m)$ divisor. This is only pseudoeffective when $m\geq 0$ and $m\geq t-d_i$.
\end{proof}

As a consequence the graded linear series $W_{\bullet}^t$ is contained in a complete linear series on the blow-up $\rho: \widetilde{E}\rightarrow E$ of the points $p_i\deq\overline C_i\cap E$, viewed as the toric variety $X_{\Sigma}$ that is the blow-up of $\PP^{n-1}$ at $n$ torus invariant points. Our goal is to show that asymptotically in terms of NObodies these two graded linear series coincide when $t$ varies.
\begin{proposition}\label{prop:equalslice}
	Let $d_1,\ldots,d_n>0$ and $D=\sum_{i=1}^nd_ih_i$. For all $t\in(0,d_1+\ldots+d_n)$ we have
	\[\inob{x}{D}_{\nu_1=t}=\nob{\widetilde Y_{\bullet}}{t\rho^*H-\sum\nolimits_{i=1}^n(t-d_i)E_i},\]
	where $Y_{\bullet}$ is a very general linear flag in $E$, $\widetilde Y_{\bullet}$ is its strict transform in $\widetilde E$, $H$ is the hyperplane class in $\bb P^{n-1}$ and $E_i$ is the exceptional divisors of $\rho$ over $p_i$.
\end{proposition}
\begin{proof}
	The very general choice of $Y_{\bullet}$ is implicit in the definition of $\inob{x}{D}$. We will see in the proof that any linear flag that avoids the points $p_i$ will suffice.
	
	We explain how Lemma \ref{lem:nonakamaye} gives the ``$\subseteq$'' inclusion. If $D'$ is an effective $\bb R$-divisor, numerically equivalent to $D$ and such that ${\rm mult}_xD'=t$, then the strict transform in $\overline X$ satisfies that $\overline{D'}\cap E$ is a degree $t$ effective divisor on $\bb P^{n-1}$ that passes through $p_i$ with multiplicity at least $t-d_i$ by Lemma \ref{lem:nonakamaye}, and has valuation $(\nu_2(D'),\ldots,\nu_n(D'))$ with respect to the linear flag $E=\bb P^{n-1}=Y_1\supset\ldots\supset Y_n$.
	The strict transform of this divisor in $\widetilde E$ determines an element of the complete linear system of $t\rho^*H-\sum_i(t-d_i)E_i$ with the same valuation with respect to the strict transform flag.
	
	For the reverse inclusion it is sufficient to prove that the convex bodies have the same Euclidean volume. We have the ``$\leq$'' inequality for each $t$.	
	It is clear that if we let $t$ vary in $[0,d_1+\ldots+d_n]$, 
	the volumes of both sides vary continuously.
	It is then enough to prove that 
	\[
	{\rm vol}_{\bb R^n}\big(\nob{x}{D}\big) \ = \ \int\nolimits_{t=0}^{d_1+\ldots+d_n}{\rm vol}_{\bb R^{n-1}}\big(\nob{\widetilde Y_{\bullet}}{t\rho^*H-\sum\nolimits_{i=1}^n(t-d_i)E_i}\big)dt
	\]
As $(D^n)=n!\cdot\prod_{i=1}^nd_i$, then by \cite[Theorem A]{LM09}  the Euclidean volume of $\nob{x}{D}$ is $\prod_{i=1}^nd_i$. Similar the volume being integrated on the right is actually $\frac 1{(n-1)!}\cdot\vol_{\widetilde{E}}(t\rho^*H-\sum(t-d_i)E_i)$. This can be computed with toric means since $\widetilde E$ is a toric variety. The material below is inspired by \cite[Section 3.1]{CDDGP11}. For a quick introduction to the theory of toric varieties: fans, polytope associated to a divisor, etc, see \cite{fulton93}.
			
	The fan of $\bb P^{n-1}$ is given by the simplicial fan with the $n$ rays $f_1,\ldots, f_{n}$ in $N_{\bb R}=\bb R^{n-1}$ where $f_1,\ldots, f_{n-1}$ are the standard basis and $f_n=-\sum_{i=1}^{n-1}e_i$. The fan of $\widetilde E$, the blow-up $\rho:\widetilde E\to\bb P^{n-1}$ of the $T$-invariant points $p_1,\ldots,p_n$ is obtained by refining this fan with the rays $e_i=-f_i$. The associated $T$-invariant divisors are denoted $F_i$ for the $f_i$, and $E_i$ for the $e_i$. The divisor $F_i$ is the strict transform in $\widetilde E$ of $\overline H_i\cap E$, the linear span  of $(p_j)_{j\neq i}$, and $E_i$ is the exceptional divisor of the blow-up of $p_i$. If $H$ denotes the hyperplane class in $E\simeq\bb P^{n-1}$, then we have the relations
	\[F_i=\rho^*H-\sum\nolimits_{j\neq i}E_i.\]	
	For our classes of interest we have
	\[t\rho^*H-\sum\nolimits_{i=1}^{n}(t-d_i)E_i=\sum\nolimits_{i=1}^n\frac tn F_i+\sum\nolimits_{i=1}^n(d_i-\frac tn)E_i\ .\]
	Their associated polytopes in $M_{\bb R}=N_{\bb R}^{\vee}$ are
	\[
	\begin{split}
		P_{t} \  & \deq \ \{m=(m_1,\ldots ,m_{n-1})\in M_{\RR} \ | \ \langle m,f_i\rangle \geq -\frac{t}{n},\ \langle m,e_i\rangle \geq \frac{t}{n}-d_i\}\\
		& = \ \{(m_1,\ldots ,m_{n-1})\in M_{\RR} \ | \ m_i\geq -\frac{t}{n},\  \sum m_i\leq \frac{t}{n},\ m_i\leq d_i-\frac{t}{n},\ \sum m_i\geq \frac{t}{n}-d_n\}
	\end{split}
	\]
	We have ${\rm vol}_{\bb R^{n-1}}\big(\nob{\widetilde Y_{\bullet}}{t\rho^*H-\sum_i(t-d_i)E_i}\big)=\frac 1{(n-1)!}\vol_{\widetilde{E}}(t\rho^*H-\sum_i(t-d_i)E_i)={\rm vol}_{\bb R^{n-1}}\big(P_t\big)$. 
	After the translations $m_i\mapsto m_i+\frac tn$, we have that $P_t$ has the same volume as
	\[
	P'_t\deq\bigl\{(m_1,\ldots,m_{n-1})\in\bb R^{n-1}\ \mid\ 0\leq m_i\leq d_i,\ t-d_n\leq \sum\nolimits_{i=1}^{n-1}m_i\leq t\bigr\}\ .
	\]
	Let
	\[
	A(t)\deq{\rm vol}_{\bb R^{n-1}}\biggl(\bigl\{(m_1,\ldots,m_{n-1})\in\bb R^{n-1}\ \mid\ 0\leq m_i\leq d_i,\ \sum\nolimits_{i=1}^{n-1}m_i\leq t\bigr\}\biggr)\ .
	\]
	We have ${\rm vol}_{\bb R^{n-1}}\big(P'_t\big)=A(t)-A(t-d_n)$. As $A(t)=0$ if $t\leq 0$ and $A(t)=\prod_{i=1}^{n-1}d_i$ for $t\geq d_1+\ldots+d_{n-1}$, then
	\[
	\int\nolimits_{t=0}^{d_1+\ldots+d_n}{\rm vol}_{\bb R^{n-1}}\big(P'_t\big)dt=\int\nolimits_{t=0}^{d_1+\ldots+d_n}(A(t)-A(t-d_n))dt=\int\nolimits_{t=d_1+\ldots+d_{n-1}}^{d_1+\ldots+d_n}A(t)dt=\prod\nolimits_{i=1}^nd_i\ .
	\]
The conclusion follows.
\end{proof}
\begin{remark}
	The function $A(t)$ is quasipolynomial, changing branches whenever $t$ is the sum of coordinates of a vertex of the box $\prod_{i=1}^{n-1}[0,d_i]$. The argument above implies that for all (possibly empty) $I\subseteq\{1,\ldots,n\}$ there exists a vertex in $\nob{x}{D}$ with $\nu_1$-coordinate $\sum_{i\in I}d_i$. The case $n=3$ of Theorem \ref{thm:smalldimensions} shows that these are not all the vertices of $\nob{x}{D}$ in general.
\end{remark}

\section{Small dimensions}

When $n=1$, then $\inob{x}{L}=[0,\deg L]$. The cases $n=2$ and $n=3$ are the subject of Theorem \ref{thm:smalldimensions}.

\begin{proof}[{\bf Proof of Theorem \ref{thm:smalldimensions}}]
When $n=2$, then $E=\bb P^1$, and $E_i=p_i$ are the two torus-fixed points. 
By Proposition \ref{prop:equalslice}, we are interested in degree $t$ divisors with multiplicity at least $t-d_i$ at $p_i$. This should be phrased as multiplicity at least $\max\{t-d_i,0\}$ at $p_i$. For $n\geq 3$ the distinction was irrelevant, essentially because if $t\geq 0$, then $\sigma^*D$ and $\sigma^*D+tE$ have isomorphic global sections when $\sigma$ is a blow-up with exceptional divisor $E$ of a point on a variety of dimension at least $2$.
It follows that $\inob{x}{L}_{\nu_1=t}$ is $\nob{Y_{\bullet}}{tY_2-\max\{t-d_1,0\}p_1-\max\{t-d_2,0\}p_2}$ where $Y_{\bullet}$ is the flag $E=\bb P^1=Y_1\supset Y_2$ and $Y_2$ is a (very) general point. This is the line segment $[0,t-\max\{t-d_1,0\}-\max\{t-d_2,0\}]$. These segments are the vertical slices of the claimed isosceles trapezoid. Other proofs when $n=2$ can be obtained by Zariski decompositions on surfaces via \cite[Theorem 6.4]{LM09}, or making use of Lemma \ref{lem:nonakamaye} and the ideas developed in \cite[Proposition 4.2]{KL18}.
\smallskip

Let now $n=3$. We first show that the convex hull of the 9 points in the conclusion has Euclidean volume $d_1d_2d_3$, same as the volume of $\nob{x}{L}$. To see this, note that we can decompose the convex hull into the three convex bodies below:
\begin{enumerate}
	\item A tetrahedron as in Theorem \ref{thm:main} for $n=3$, but scaled by $d_3$. Its vertices are $(0,0,0)$, $(d_3,d_3,0)$, $(2d_3,0,2d_3)$, $(3d_3,0,0)$. Its Euclidean volume is $d_3^3$.
	\smallskip
	
	\item A tringular prism. One triangular face is shared with the tetrahedron and has vertices
	\[(d_3,d_3,0),\ (2d_3,0,2d_3),\ (3d_3,0,0)\]
	The opposite parallel face has corresponding vertices
	\[(d_1+d_2-d_3,d_3,0),\ (d_1+d_2,0,2d_3),\ (d_1+d_2+d_3,0,0)\]
	One of its parallelogram faces is in the plane $\nu_3=0$ where it identifies with the parallelogram with vertices $(d_3,d_3)$, $(d_1+d_2-d_3,d_3)$, $(3d_3,0)$, $(d_1+d_2+d_3,0)$. The area of the parallelogram is $(d_1+d_2-2d_3)\cdot d_3$. The opposite parallel edge to this parallelogram is at height $2d_3$. Therefore the volume of this prism is $(d_1+d_2-2d_3)d_3^2$.
	\smallskip
	
	\item A prism over the isosceles trapezoid with vertices
	\[(2d_3,0,2d_3),\ (d_1+d_2,0,2d_3),\ (d_1+d_3,0,d_2+d_3),\ (d_2+d_3,0,d_2+d_3)\]
	inside the plane $\nu_2=0$ and with opposite face at ``height'' $d_3$ (in the perpendicular $\nu_2$-direction) the trapezoid with corresponding vertices
	\[(d_3,d_3,0),\ (d_1+d_2-d_3,d_3,0),\ (d_1,d_3,d_2-d_3),\ (d_2,d_3,d_2-d_3)\]
	The area of the trapezoid is $(d_1-d_3)(d_2-d_3)$, and the volume of this prism is $(d_1-d_3)(d_2-d_3)d_3$. The prism shares a parallelogram face with the triangular prism and an edge with both the triangular prism and the tetrahedron.
\end{enumerate}

The volumes above add up to $d_1d_2d_3$ as claimed. Denote this convex hull by $C$. For volume reasons, it is sufficient to prove that $C\subseteq \nob{x}{L}$, equivalently that the 9 vertices of $C$ are contained in $\nob{x}{L}$.

The vertex $(0,0,0)$ is the valuation of any effective divisor that does not pass through $x$. The other easy vertex $(d_1+d_2+d_3,0,0)$ is the valuation with respect to a general linear flag $Y_{\bullet}$ in $E$ of the divisor $d_1H_1+d_2H_2+d_3H_3$ numerically equivalent to $L$. As in Section 3, the $H_i$ are the ``coordinate'' hypersurfaces through $x$. For the remaining points we use the interpretation of the vertical slices in Proposition \ref{prop:equalslice}.

In what remains, denote by $(d;a_1,a_2,a_3)$ the linear system of degree $d$ curves in $\bb P^2$ with multiplicity at least $a_i$ at the $p_i$. Note that if for example $a_1$ is negative, then $(d;a_1,a_2,a_3)=(d;0,a_2,a_3)$. The linear system identifies essentially by strict transforms with the complete system of $d\rho^*H-\sum_ia_iE_i$ on $\widetilde E$. 
The valuation with respect to $E=Y_1\supset Y_2$ for elements of $(d;a_1,a_2,a_3)$ is the same as the valuation with respect to the strict transform flag for the corresponding sections of $d\rho^*H-\sum_ia_iE_i$. Note that the strict transform flag is well-defined if $Y_2\neq p_i$.

Consider the following curves in $\bb P^2$:
\begin{itemize}
	\item $F_i$ is the line through $p_j$ and $p_k$ for all $\{i,j,k\}=\{1,2,3\}$. For example $F_1\in(1;0,1,1)$ and its valuation is $(0,0)$ with respect to our flag. 
	\smallskip
	
	\item $Y_2$ is the line in the flag, a very general line in $\bb P^2$. We have $Y_1\in(1;0,0,0)$ and its valuation is $(1,0)$.
	\smallskip
	
	\item $\ell_i$ the line joining the flag point $Y_3$ with $p_i$. For instance $\ell_1\in(1;1,0,0)$ and its valuation is $(0,1)$.
	\smallskip
	
	\item The conic $Q$ through $p_1$, $p_2$, $p_3$ and tangent to $Y_2$ at the point $Y_3$. It is an element of $(2;1,1,1)$ with valuation $(0,2)$.
\end{itemize}

The following list of divisors in $\bb P^2$, satisfy the conditions from Proposition \ref{prop:equalslice}, and provide modulo the above flag valuation the remaining vertices of our body $\inob{x}{L}$.

\begin{enumerate}
	\item $d_3\cdot Y_2\in (d_3;0,0,0)=(d_3;d_3-d_1,d_3-d_2,0)$ has valuation $(d_3,0)$ with respect to our flag in $\bb P^2$, and shows that $(d_3,d_3,0)\in\inob{x}{L}$.
	\item $d_3Y_2+(d_2-d_3)\ell_3\in(d_2;0,0,d_2-d_3)=(d_2;d_2-d_1,0,d_2-d_3)$ has valuation $(d_3,d_2-d_3)$ and determines the vertex $(d_2,d_3,d_2-d_3)\in\inob{x}{L}$.
	\item $d_3Y_2+(d_1-d_2)F_1+(d_2-d_3)\ell_3\in(d_1;0,d_1-d_2,d_1-d_3)$ has valuation $(d_3,d_2-d_3)$ and determines the vertex $(d_1,d_3,d_2-d_3)\in\inob{x}{L}$.
	\item $d_3Y_2+(d_2-d_3)F_2+(d_1-d_3)F_1\in(d_1+d_2-d_3;d_2-d_3,d_1-d_3,d_1+d_2-2d_3)$ has valuation $(d_3,0)$ and determines the vertex $(d_1+d_2-d_3,d_3,0)\in\inob{x}{L}$.	
	\item $d_3Q+(d_2-d_3)\ell_3\in (d_2+d_3;d_3,d_3,d_2)\subseteq(d_2+d_3;-d_1+d_2+d_3,d_3,d_2)$ has valuation $(0,d_2+d_3)$ and determines the vertex $(d_2+d_3,0,d_2+d_3)\in\inob{x}{L}$.
	\item $d_3Q+(d_1-d_2)F_1+(d_2-d_3)\ell_3\in(d_1+d_3;d_3,d_1-d_2+d_3,d_1)$ has valuation $(0,d_2+d_3)$ and determines the vertex $(d_1+d_3,0,d_2+d_3)\in\inob{x}{L}$.
	\item $d_3Q+(d_2-d_3)F_2+(d_1-d_3)F_1\in(d_1+d_2;d_2,d_1,d_1+d_2-d_3)$ has valuation $(0,2d_3)$ and determines the vertex $(d_1+d_2,0,2d_3)\in\inob{x}{L}$.
\end{enumerate}
These conclude the proof.
\end{proof} 

\begin{remark}
	Our strategy of proof for Theorem \ref{thm:smalldimensions} boils down to a guess-and-check. The coordinates of the vertices of the generic iNObody were found with a different argument in \cite{FL25fzproducts}. 
\end{remark}

\section{Balanced box-products in arbitrary dimension}

In this section we prove Theorem \ref{thm:main}. The strategy will be similar to that of Theorem \ref{thm:smalldimensions}, although the argument here is existential, not constructive.

\begin{lemma}\label{lemma:computing}
	For any $n\geq d\geq 1$ we have the following inclusion:
	\[
	d\cdot {\bf e}_d \ \in \ \nob{\widetilde Y_{\bullet}}{d\rho^*H-\sum\nolimits_{i=1}^n(d-1)E_i} \ ,
	\]
	where ${\bf e}_1,\ldots,{\bf e}_{n-1}$ are the standard basis of $\bb R^{n-1}$, while ${\bf e}_n=(0,\ldots ,0)\in \RR^{n-1}$ is the origin. 
\end{lemma}

\begin{proof}
	Consider homogeneous coordinates $x_1,\ldots,x_n$ on $\bb P^{n-1}$.
	Let $Q\deq x_1\cdots x_n$. It has degree $n$ and multiplicity $n-1$ at the $p_i$, hence determines an element of $|n\rho^*H-\sum_{i=1}^n(n-1)E|$. When $Y_{\bullet}: \bb P^{n-1}=Y_1\supset\ldots\supset Y_n$ is a general linear flag so that $Q(Y_n)\neq 0$, then $\nu_{\widetilde Y_{\bullet}}(\{Q=0\})=(0,\ldots,0)$. This settles the case $n=d$.
	We show the case $n>d$ by induction.
	Before that, consider the incomplete linear system
	\[
	\Lambda_d\deq \bb C\langle x_I\ \mid\ I\subset\{1,\ldots,n\},\ |I|=d\rangle 
	\]
	where $x_I\deq\prod_{i\in I}x_i$.
	It is easy to see that $\Lambda_d$ identifies naturally with a subspace of $|d\rho^*H-\sum_{i=1}^n(d-1)E_i|$, for example by applying the degree $n-d>0$ partial derivatives to $Q$.
	
	We claim that for a general linear flag term $Y_{d}\simeq \bb P^{n-d} \subset\bb P^{n-1}$, the morphism induced by restriction
	\[\psi_d:\Lambda_d\hookrightarrow H^0\bigl(\bb P^{n-1},\cal O_{\bb P^{n-1}}(d)\bigr)\to H^0\bigl(Y_{d},\cal O_{\bb P^{n-d}}(d)\bigr)\]
	is an isomorphism. This suffices as then $\psi_d$ is surjective, and so there exists $Q_d\in\Lambda_d$ that restricts to $Y_d$ as $f^d$, where $f=0$ is the linear equation of $Y_{d+1}$ in $Y_d$. Then $\nu_{\widetilde Y_{\bullet}}(Q_d)=d\cdot{\bf e}_d$ as needed.
	
	For the claim, it suffices to show that $\psi_d$ is injective when $Y_d$ is general, as both sides of $\psi_d$ have dimension ${n\choose d}$. Let $z_1,\ldots,z_{n-d+1}$ be homogeneous coordinates on $\bb P^{n-d}$. Let $Y_d\subset\bb P^{n-1}$ be parameterized by \[\bb P^{n-d}\ni [z_1:\ldots:z_{n-d+1}]=z\mapsto [f_1(z):\ldots:f_n(z)]\in\bb P^{n-1}\]
	where the $f_i$ are general linear forms on $\bb P^{n-d}$. For $I\subset\{1,\ldots,n\}$, denote by $f_I=\prod_{i\in I}f_i$. We want to show that the equation of homogenous forms in $\bb C[z_1,\ldots,z_{n-d+1}]$ with complex coefficients
	\[
	\sum_{|I|=d}c_If_I=0\ ,
	\]
	has only one solution, the trivial one, when $c_I=0$ for all $I$.
	Equivalently, $(f_I(z))_{|I|=d}$ are linearly independent. 
	This we prove by induction on $n\geq d+1$.
	
	Assume first $n=d+1$. Then $Y_d=Y_{n-1}\simeq\bb P^1$. Due to our generality condition, we may assume after scaling that $f_i=z_1-r_iz_2$ for distinct $r_i\in\bb C$. The linear independence of the $f_I$ is equivalent to the linear independence of $\prod_{j\neq i}(z-r_j)$ in $\bb C[z]$. This follows easily by evaluating at the $r_i$, using that they are distinct.
	
	Assume now $n>d+1$ and that the claim is valid for $n-1$. By generality, after scaling we may assume $f_n(z_1,\ldots,z_{n-d+1})=\sum_{i=1}^{n-d}\alpha_iz_i-z_{n-d+1}$. Consider the morphism 
	\[
	\varphi:\bb C[z_1,\ldots,z_{n-d+1}]\to\bb C[z_1,\ldots,z_{n-d}]
	\]
	 determined by $\varphi(z_i)=z_i$ for $i\in\{1,\ldots,n-d\}$ and $\varphi(z_{n-d+1})=\sum_{i=1}^{n-d}\alpha_iz_i$. It maps $f_n$ to $0$. Set $g_i\deq\varphi(f_i)$ for $i\in\{1,\ldots,n-1\}$. These are general linear forms in $z_1,\ldots,z_{n-d}$. Also denote $g_I=\prod_{i\in I}g_i$. By the induction hypothesis, the $(g_I)_{|I|=d,\ I\subset\{1,\ldots,n-1\}}$ are linearly independent. 
	Assume $\sum_{|I|=d}c_If_I=0$. Then $\sum_{|I|=d,\ n\not\in I}c_Ig_I=0$ and so $c_I=0$ for all $I$ that do not contain $n$. Now repeat the argument for all other indices in $\{1,\ldots,n\}$. 
\end{proof}

\begin{proof}[{\bf Proof of Theorem \ref{thm:main}}]
	Since $\inob{x}{L}$ is convex of Euclidean volume $1$, same as the simplex in the conclusion, it is sufficient to prove that all the vertices of the simplex belong to $\inob{x}{L}$. 
	The origin is the valuation of any effective divisor that does not pass through $x$. On the other extreme, $(n,0,\ldots,0)$ is the valuation of $H_1+\ldots+H_n$ with respect to a very general linear flag in $E$.
	Lemma \ref{lemma:computing} gives all other vertices by Proposition \ref{prop:equalslice}.
\end{proof}

\section{Questions and open problems}

\begin{problem}\label{prob:one}
Let $X=C_1\times\ldots\times C_n$ and $L=L_1\boxtimes\ldots\boxtimes L_n$ a boxed-product with $\deg L_i=d_i$. Let $x\in X$ be any point, then compute the convex set $\Delta_x(L)\subseteq \RR^n_{\geq 0}$, generalizing the case $n=2,3$ from Theorem \ref{thm:smalldimensions}.
\end{problem}
We expect $\Delta_x(L)$ to be a polytope, but as the case $n=3$ shows, guessing does not seem a trivial game. The techniques from \cite{FL25jac} of computing birational Fujita--Zariski decompositions may apply in higher dimension as well. The toric variety we are likely to have to work on is the blow-up of all the torus invariant linear strata in $\bb P^{n-1}$, resolving the Cremona transformation. This is an interesting space in itself. It can be interpreted as the toric variety associated to the root system $A_{n-1}$ and to the permutohedron, or the Losev--Manin moduli space $\overline L_n$. See \cites{LM00,BB11,Huh18}. The combinatorics appear involved.

\begin{problem}
Let $C$ be a smooth projective curve and $X=C^n$. Study the shape of the generic iNObody at any point $x\in X$ for natural non-boxed product type divisors.
\end{problem}
When each $C=\PP^1$, every divisor class on $X$ is of boxed product type. However, when $g(C)\geq 1$, the diagonal $\Delta\subseteq C\times C$ is already not of boxed product type. Here even understanding which classes $d(h_1+h_2)-t\Delta$ are effective is an open question in general. See \cite[Section 1.5.B]{laz04} for the history of the problem, \cites{cknagata,Ross} for its relation to the Nagata conjecture, and \cite{fm2} for more recent results. More generally, one can construct various natural morphisms $C^n\to{\rm Jac}(C)$, and then the pullback of the theta polarization is an interesting class that is usually not of boxed-product type.

\begin{problem}
Let $\Sigma$ be a complete $n$-dimensional fan such that it's associated toric variety  $X_{\Sigma}$ is smooth. For any point $x\in X$ compute the generic global iNObody $\Delta_x(X_{\Sigma})\subseteq N^1(X_{\Sigma})\times \RR^n$.
\end{problem}

The problem is already very difficult for general $x\in X$ on toric surfaces where Seshadri constants are not computed (see \cites{Eck08,ItoSeshadriToric,DRL22,LU23}). 
It is also interesting for $T$-invariant points, because of the genericity condition on the flag which means that the flag terms are not $T$-invariant.
Examples show that these convex sets constructed from general linear flags have simpler shapes than for $T$-invariant flags on the exceptional divisor of the blow-up of the point. For example, on the toric surface $X=\bb P^1\times\bb P^1$ consider the polarization $L=\cal O(1,1)$ and a point $x\in X$. Then $\inob{x}{L}$ is the triangle with vertices $(0,0)$, $(1,1)$, and $(2,0)$. But there are only two choices of $T$-invariant flags in the exceptional divisor on the blow-up of $x$. They correspond to the two rulings on a smooth quadric surface. Either choice leads to the parallelogram with vertices $(0,0)$, $(1,0)$, $(1,1)$, and $(2,1)$, a different body from the triangle above. Considering the toric ideas in our proofs, the previous question seems related to the following:

\begin{problem}Let $f:X_{\Sigma}\to\bb P^{n}$ be a composition of toric blow-ups. If $Y_{\bullet}$ is a (very) general complete linear flag on $\bb P^{n}$, then its strict transform flag $Z_{\bullet}$ in $X_{\Sigma}$ is well-defined. Compute the global NObody $\Delta_{Z_{\bullet}}(X_{\Sigma})$.
\end{problem}

In \cite{FL25gen} we explain how these NObodies with respect to a very general choice of a linear flag on $X_{\Sigma}$ satisfy properties inspired by the theory of generic initial ideals. They usually do not fully describe the body when our ambient space is a toric variety as above. The goal is to find algorithms that construct the convex body $\Delta_{Z_{\bullet}}(X_{\Sigma})$ from the convex data of the complete fan $\Sigma$. In the case of the blow-up of $\PP^{n-1}$ in $n$ points in general linear position we were able to construct the body directly using its rich geometry.

\bibliographystyle{alpha}
\bibliography{Positive}
\end{document}